\documentclass{amsart}

\newtheorem{theorem}{Theorem}[section]
\newtheorem{lemma}[theorem]{Lemma}

\theoremstyle{definition}

\newtheorem{remark}[theorem]{Remark}

\usepackage{amscd,amssymb}
\begin{document}

\title[Generalized Nowicki conjecture]
{Generalized Nowicki Conjecture}
\author[Vesselin Drensky]
{Vesselin Drensky}
\address{Institute of Mathematics and Informatics,
Bulgarian Academy of Sciences,
1113 Sofia, Bulgaria}
\email{drensky@math.bas.bg}


\begin{abstract}
Let $B$ be an integral domain over a field $K$ of characteristic 0.
The derivation $\delta$ of $B[Y_d]=B[y_1,\ldots,y_d]$ is elementary if $\delta(B)=0$ and $\delta(y_i)\in B$, $i=1,\ldots,d$.
Then the elements $u_{ij}=\delta(y_i)y_j-\delta(y_j)y_i$, $1\leq i<j\leq d$, belong to the algebra $B[Y_d]^{\delta}$
of constants of $\delta$ and it is a natural question whether the $B$-algebra $B[Y_d]^{\delta}$ is generated by all $u_{ij}$.
In this paper we consider the special case of $B=K[X_d]=K[x_1,\ldots,x_d]$.
If $\delta(y_i)=x_i$, $i=1,\ldots,d$, this is the Nowicki conjecture
from 1994 which was confirmed in several papers based on different methods. The case
$\delta(y_i)=x_i^{n_i}$, $n_i>0$, $i=1,\ldots,d$, was handled by Khoury in the first proof of the
Nowicki conjecture given by him in 2004.
As a consequence of the proof of Kuroda in 2009 if
$\delta(y_i)=f_i(x_i)$, for any nonconstant polynomials $f_i(x_i)$, $i=1,\ldots,d$,
then $B[Y_d]^{\delta}=K[X_d,Y_d]^{\delta}$ is generated by $X_d$ and $U_d=\{u_{ij}=f_i(x_i)y_j-y_if_j(x_j)\mid 1\leq i<j\leq d\}$.
In the present paper we have found a presentation of the algebra
\[
K[X_d,Y_d]^{\delta}=K[X_d,U_d\mid R=S=0],
\]
\[
R=\{r(i,j,k,l)\mid 1\leq i<j<k<l\leq d\}, S=\{s(i,j,k)\mid 1\leq i<j<k\leq d\},
\]
and a basis of $K[X_d,Y_d]^{\delta}$ as a vector space. As a corollary we have shown that the defining relations
$R\cup S$ form the reduced Gr\"obner basis of the ideal which they generate with respect to a specific admissible order.
This is an analogue of the result of Makar-Limanov and the author in their proof of the Nowicki conjecture in 2009.
\end{abstract}

\maketitle

\section{Introduction}
In the present paper we consider only commutative algebras over a field $K$ of characteristic 0.
A linear operator $\delta$ of an algebra $A$ is a derivation if
it satisfies the Leibniz rule
\[
\delta(a_1a_2)=\delta(a_1)a_2+a_1\delta(a_2),\quad a_1,a_2\in A.
\]
The kernel $A^{\delta}$ of $\delta$ is the algebra of constants of $\delta$.
Let $B$ be an integral domain over $K$ and let $A=B[Y_d]=B[y_1,\ldots,y_d]$.
The derivation $\delta$ of $A$ is elementary if $\delta(B)=0$ and $\delta(y_i)\in B$, $i=1,\ldots,d$.
Then the determinants
\begin{equation}\label{constants determinants}
u_{ij}=\left\vert\begin{matrix}\delta(y_i)&\delta(y_j)\\
y_i&y_j\\
\end{matrix}\right\vert=\delta(y_i)y_j-\delta(y_j)y_i,\quad 1\leq i<j\leq d,
\end{equation}
belong to $B[Y_d]^{\delta}$ and it is a natural question whether the $B$-algebra $B[Y_d]^{\delta}$ is generated by
the elements (\ref{constants determinants}).
In the sequel we assume that $B=K[X_d]=K[x_1,\ldots,x_d]$ (and $\delta(B)=0$).
In the special case
\begin{equation}\label{Nowicki derivation}
\delta(y_i)=x_i,\quad i=1,\ldots,d,
\end{equation}
the finite generation of $A^{\delta}=B[Y_d]^{\delta}=K[X_d,Y_d]^{\delta}$ follows
from a more general result of Weitzenb\"ock \cite{W} in 1932.
In 1994 Nowicki \cite{N1} conjectured that for $\delta$ from (\ref{Nowicki derivation})
the algebra $K[X_d,Y_d]^{\delta}$ is generated by $X_d$ and
\[
u_{ij}=x_iy_j-x_jy_i,\quad 1\leq i<j\leq d.
\]
This was confirmed in several papers based on different methods, see, e.g., \cite{Ku} and \cite{D} for details.

In the first proof of the Nowicki conjecture given in his Ph.D. thesis Khoury \cite{K1, K3} made one more step and
established a result which gives an answer to a generalization of the Nowicki conjceture.
If $n_i$, $i=1,\ldots,d$, are positive integers,
and the derivation $\delta$ of $B[Y_d]=K[X_d,Y_d]$ is defined by
\[
\delta(y_i)=x_i^{n_i},\quad i=1,\ldots,d,
\]
then the algebra $K[X_d,Y_d]^{\delta}$ of constants of $\delta$ is generated again by $X_d$ and 
\[
u_{ij}=x_i^{n_i}y_j-x_j^{n_j}y_i,\quad 1\leq i<j\leq d.
\]
The most general result in this direction belongs to Kuroda \cite{Ku}.

\begin{theorem}
\label{theorem of Kuroda}
Let $B$ be an integral domain over $K$ and let $\delta$ be an elementary
derivation of $B[Y_d]$ such that $f_i=\delta(y_i)$, $i=1,\ldots,d$, are algebraically independent over $K$.
If $B$ is flat over $K[F_d]=K[f_1,\ldots,f_d]$, then the $B$-algebra $B[Y_d]^{\delta}$ is generated by
$u_{ij}=f_iy_j-f_jy_i$, $1\leq i<j\leq d$.
\end{theorem}

It seems that it is difficult to find further generalizations.
Khoury \cite{K2} showed that the algebra of constants of the elementary derivation $\delta$ of $B[Y_4]$, $B=K[X_3]$,
for $F_4=\{x_1^2,x_2^2,x_3^2,x_2x_3\}$ is finitely generated but cannot be generated by expressions which are linear in $Y_4$.
Also, many of the modern counterexamples to the Fourteenth Hilbert problem are constructed in terms of elementary derivations,
see, e.g., the surveys by Freudenburg \cite{F} and Nowicki \cite{N2}.

It is easy to see that Theorem \ref{theorem of Kuroda} holds for the elementary derivation $\Delta$ of $B[Y_d]$, $B=K[X_d]$, and
\[
\Delta(y_i)=f_i(x_i),\quad i=1,\ldots,d,
\]
where $f_i(x_i)$ is a nonconstant polynomial in the variable $x_i$, $i=1,\ldots,d$.
But there is an essential difference between the Nowicki conjecture and this result.
The Nowicki conjecture is equivalent to a statement of classical invariant theory.
When the polynomials $f_i(x_i)$ are not linear, we cannot see how to restate the result in the language of invariant theory.
In the present paper we apply the methods developed in our proof with Makar-Limanov \cite{DML} of the Nowicki conjecture
and find a presentation of the algebra $K[X_d,Y_d]^{\Delta}$ and a basis of $K[X_d,Y_d]^{\Delta}$ as a vector space.
As a consequence we show that our defining relations
form the reduced Gr\"obner basis of the ideal which they generate with respect to a specific admissible order.

\section{Generators of the algebra of constants}

We shall need the following easy lemma. We include the proof for self-containedness of the exposition.

\begin{lemma}\label{free module}
Let $f_i=f_i(x_i)$, $i=1,\ldots,d$, be nonconstant polynomials.
Then $K[X_d]$ is a free $K[f_1,\ldots,f_d]$-module.
\end{lemma}

\begin{proof}
Let $m_i=\deg(f_i)$ be the degree of $f_i$. Every monomial $x_i^m$ can be written as a linear combination of
polynomials $q_n(x_i)f_i^n(x_i)$, $\deg(q_n)<m_i$, $n=0,1,2,\ldots$, and hence as a $K[f_1,\ldots,f_d]$-module
$K[X_d]$ is generated by the monomials
\[
x_1^{r_1}\cdots x_d^{r_d},\quad 0\leq r_i<m_i,\quad i=1,\ldots,d.
\]
The leading terms of the products
\begin{equation}\label{leading term lexicographical}
x_1^{r_1}f_1^{n_1}(x_1)\cdots x_d^{r_d}f_d^{n_d}(x_d)
\end{equation}
with respect to the lexicographic order are equal to $x_1^{n_1m_1+r_1}\cdots x_1^{n_dm_d+r_d}$
and are pairwise different. Hence the polynomials (\ref{leading term lexicographical}) are linearly independent
which implies that $K[X_d]$ is a free $K[f_1,\ldots,f_d]$-module.
\end{proof}

\begin{theorem}\label{generalized conjecture}
Let $f_i=f_i(x_i)$, $i=1,\ldots,d$, be nonconstant polynomials in one variable and let $\Delta$
be the derivation of $K[X_d,Y_d]$ defined by
\[
\Delta(y_i)=f_i(x_i),\quad\Delta(x_i)=0,\quad i=1,\ldots,d.
\]
Then the algebra $K[X_d,Y_d]^{\Delta}$ of constants of $\Delta$ is generated by $X_d$ and
the determinants from \text{\rm (\ref{constants determinants})}
\begin{equation}\label{generators for Delta}
u_{ij}=f_i(x_i)y_j-f_j(x_j)y_i,\quad 1\leq i<j\leq d.
\end{equation}
\end{theorem}

\begin{proof} By \cite[Corollary 6.6, p. 165]{E} if
a $B$-module $M$ is finitely generated then it is flat if and only if it is a summand of a free $B$-module.
Hence by Lemma \ref{free module} $K[X_d]$ is a flat $K[f_1,\ldots,f_d]$-module. Obviously
$f_i(x_i)$, $i=1,\ldots,d$, are algebraically independent over $K$ and Theorem \ref{theorem of Kuroda}
immediately gives that the algebra $K[X_d,Y_d]^{\delta}$ is generated by $X_d$ and the polynomials
(\ref{generators for Delta}).
\end{proof}

\begin{remark}
If some of the polynomials $f_i(x_i)$ in Theorem \ref{generalized conjecture} is a constant, then
the description of $K[X_d,Y_d]^{\Delta}$ is trivial. Let, for example, $f_1(x_1)=\alpha\in K\setminus\{0\}$.
We replace the variables $Y_d$ by $Z_d$, where
\[
z_1=y_1,\quad z_i=\alpha y_i-f_i(x_i)y_1,\quad i=2,\ldots,d.
\]
Then the definition of $\Delta$ becomes
\[
\Delta(z_1)=\alpha,\quad \Delta(z_i)=0,\quad i=2,\ldots,d,\quad \Delta(X_d)=0.
\]
Since $\alpha\not=0$, we obtain that $K[X_d,Z_d]^{\Delta}=K[X_d,z_2,\ldots,z_d]$.
\end{remark}

\section{The main result}

In this section we follow our paper with Makar-Limanov \cite{DML}
and use the methods developed there. Since we shall work with Gr\"obner bases,
we refer, e.g., to the book by Adams and Loustaunau \cite{AL} for a background on the topic.
We fix the degrees $m_1,\ldots,m_d$ of the polynomials $f_1(x_1),\ldots,f_d(x_d)$
and the set
\[
U_d=\{u_{ij}\mid 1\leq i<j\leq d\},
\]
where the elements $u_{ij}$ are defined in (\ref{generators for Delta}).

\begin{lemma}\label{relations between generators}
The subsets $X_d$ and $U_d$ of $K[X_d,Y_d]$ satisfy the relations $R=S=0$,
%\[
%r(i,j,k,l)=0,\quad r\in R,\quad s(i,j,k)=0,\quad s\in S,
%\]
where
\begin{equation}\label{relations between U}
R=\{r(i,j,k,l)= u_{ij}u_{kl} - u_{ik}u_{jl}+ u_{il}u_{jk}\mid
1\leq i<j<k<l\leq d\},
\end{equation}
\begin{equation}\label{relations between X and U}
S=\{s(i,j,k)= f_i(x_i)u_{jk} - f_j(x_j)u_{ik} + f_k(x_k)u_{ij}\mid 1\leq i<j<k\leq d\}.
\end{equation}
\end{lemma}

\begin{proof}
The annihilating of the relations (\ref{relations between U}) and (\ref{relations between X and U})
in $K[X_d,Y_d]$ can be verified directly.
Instead, the expansions of the determinants
\[
r(i, j, k, l) = \left\vert
\begin{array}{cccc}
 f_i(x_i) & f_j(x_j) & f_k(x_k) & f_l(x_l) \\
 y_i & y_j & y_k & y_l \\
 f_i(x_i) & f_j(x_j) & f_k(x_k) & f_l(x_l) \\
 y_i & y_j & y_k & y_l \\
 \end{array}
 \right\vert
\]
and
\[
s(i, j, k) = \left\vert
\begin{array}{ccc}
f_i(x_i) & f_j(x_j) & f_k(x_k) \\
f_i(x_i) & f_j(x_j) & f_k(x_k) \\
y_i & y_j & y_k  \\
\end{array}
\right\vert
\]
relative to the first two rows and to the first row give, respectively,
(\ref{relations between U}) and (\ref{relations between X and U}).
\end{proof}

Now we shall work in the polynomial algebra $K[X_d,U_d]$. By Theorem \ref{generalized conjecture}
there is a canonical epimorphism
\begin{equation}\label{epimorphism pi}
\pi:K[X_d,U_d]\to K[X_d,Y_d]^{\Delta}.
\end{equation}
Since there will be no misunderstanding,
we shall use the same symbols $x_i\in X_d$ and $u_{jk}\in U_d$ for the generators of the polynomial algebra $K[X_d,U_d]$
and their images under $\pi$ which generate the algebra of constants $K[X_d,Y_d]^{\Delta}$.

Our first goal is to show that the kernel of $\pi$ is generated by the relations
$r(i,j,k,l)$ and $s(i,j,k)$ from (\ref{relations between U}) and (\ref{relations between X and U}).
We associate with every $u_{ij}\in U_d$ the open interval $(i,j)$. As in \cite{DML}
we define an ordering  of $K[X_d,U_d]$ called degree--interval length--lexicographic order (DILL order).
We order the monomials of $K[X_d,U_d]$ first by the degree in $X_d$ and in $U_d$,
then by the total length of the intervals associated with the participating
variables $u_{ij}$ and finally lexicographically.
If
\[
v=x_{i_1}\cdots x_{i_p}u_{j_1k_1}\cdots u_{j_qk_q}\text{ and }
v'=x_{i'_1}\cdots x_{i'_{p'}}u_{j'_1k'_1}\cdots u_{j'_{q'}k'_{q'}},
\]
where $1\leq i_1\leq\cdots\leq i_p\leq d$, $1\leq j_1\leq\cdots\leq j_q\leq d$ and
$1\leq k_a\leq k_{a+1}\leq d$ if $j_a=j_{a+1}$,
with similar restrictions on $i'_a,j'_b,k'_b$,
we define $v\succ v'$ if

(i) $p>p'$ (we compare the degrees of $v$ and $v'$ in $X_d$);

(ii) $p=p'$ and $q>q'$ (we compare the degrees of $v$ and $v'$ in $U_d$);

(iii) $p=p'$, $q=q'$ and
\[
\sum_{b=1}^q(k_b-j_b)> \sum_{b=1}^q(k'_b-j'_b)
\]
(we compare the total lengths of the intervals associated with $v$ and $v'$);

(iv) $p=p'$, $q=q'$,
\[
\sum_{b=1}^q(k_b-j_b) = \sum_{b=1}^q(k'_b-j'_b)
\]
and $\omega>\omega'$
for the $(p+2q)$-tuples
\[
\omega=(i_1,\ldots, i_p,j_1,\ldots,j_q,k_1,\ldots,q_d),\quad
\omega'=(i'_1,\ldots, i'_p,j'_1,\ldots,j'_q,k'_1,\ldots,k'_q),
\]
where $(a_1,\ldots,a_n) > (b_1,\ldots,b_n)$
if $a_1=b_1,\ldots,a_e=b_e$, $a_{e+1}>b_{e+1}$ for some $e$ (we compare lexicographically $v$ and $v'$).

Obviously the DILL-order is admissible, i.e.,
it is linear, satisfies the descending chain condition, and
if $v\succ v'$ for two monomials $v$ and $v'$,
then $vw\succ v'w$ for all monomials $w$.
For $0\not=f\in K[X_d,U_d]$ we denote by $\overline{f}$ the leading monomial of $f$.

\begin{lemma}\label{spanning set of algebra of constants}
The set of normal words in $K[X_d,U_d]$ with respect to DILL order
and modulo the relations $R\cup S$ consists of the monomials
\begin{equation}\label{normal monomials}
v=x_1^{a_1}\cdots x_d^{a_d}u_{j_1k_1}\cdots u_{j_qk_q}, \quad 1\leq j_b<k_b\leq d,
\end{equation}
such that

{\rm (i)} If $(j_b,k_b)\cap (j_c,k_c)\not=\emptyset$ for two different $u_{j_bk_b}$ and $u_{j_ck_c}$
in {\rm (\ref{normal monomials})},
then one of the intervals $(j_b,k_b)$ and $(j_c,k_c)$ is contained in the other;

{\rm (ii)} If $i\in(j_b,k_b)$ for some $u_{j_bk_b}$ in {\rm (\ref{normal monomials})}, then $a_i<m_i$.

As a vector space the algebra $K[X_d,Y_d]^{\Delta}$ is spanned by the images under $\pi$ of the products
{\rm (\ref{normal monomials})}.
\end{lemma}

\begin{proof}
The leading monomials with respect to the DILL order
of $r(i,j,k,l)$ and $s(i,j,k)$ from (\ref{relations between U}) and (\ref{relations between X and U}) are,
respectively,
\begin{equation}\label{leading monomials of relations}
u_{ik}u_{jl},\quad 1\leq i<j<k<l\leq d,\text{ and }x_j^{m_j}u_{ik},\quad 1\leq i<j<k\leq d.
\end{equation}

(i) Let $u_{j_bk_b}u_{j_ck_c}$ divide the monomial $v$, and let $j_b\leq j_c$ and if $j_b=j_c$, then $k_b\leq k_c$.
If the intervals $(j_b,k_b)$ and $(j_c,k_c)$
have a nontrivial intersection and are not contained in each other,
then $j_b<j_c<k_b<k_c$. Hence the monomial $u_{j_bk_b}u_{j_ck_c}$ is the leading term of $r(j_b,j_c,k_b,k_c)$.
In this way $v$ is not a normal word and this proves (i).

(ii) If the monomial $v$ is a normal word
and $x_i^{a_i}u_{j_bk_b}$ divides $v$ with $i\in(j_b,k_b)$, then $x_i^{a_i}u_{j_bk_b}$ is not divisible by
a monomial from (\ref{leading monomials of relations}) and hence $a_i<m_i$. This proves (ii).

Let us take some set of polynomials $W=\{w_1,\ldots,w_n\}\subset K[X_d,U_d]$
which are in the kernel of $\pi$ from (\ref{epimorphism pi}).
The images in $K[X_d,Y_d]^{\Delta}$ of the normal words in $K[X_d,U_d]$
with respect to $W$ span the vector space $K[X_d,Y_d]^{\Delta}$.
By Lemma \ref{relations between generators} the set of relations $R\cup S$
from (\ref{relations between U}) and (\ref{relations between X and U})
belong to the kernel of $\pi$. Hence $K[X_d,Y_d]^{\Delta}$ is spanned by the images of the normal words from
(\ref{normal monomials}).
\end{proof}

\begin{lemma}\label{basis of normal elements}
The images under the epimorphism $\pi$ from {\rm (\ref{epimorphism pi})}
of the normal words {\rm (\ref{normal monomials})} from Lemma {\rm\ref{spanning set of algebra of constants}}
form a basis of the vector space $K[X_d,Y_d]^{\Delta}$.
\end{lemma}

\begin{proof}
Let $v=x_1^{a_1}\cdots x_d^{a_d}u_{j_1k_1}\cdots u_{j_qk_q}\in K[X_d,U_d]$ be a normal word.
Then
\begin{equation}\label{image of normal word}
\pi(v)=x_1^{a_1}\cdots x_d^{a_d}\prod_{b=1}^q(f_{j_b}(x_{j_b})y_{k_b}-f_{k_b}(x_{k_b})y_{j_b}).
\end{equation}
We shall follow the proof of \cite[Step 3 in the proof of Theorem 5]{DML}.
We compare the monomials in $K[X_d,Y_d]=K[x_1,y_1,\ldots,x_d,y_d]$ lexicographically:
\[
x_1^{a_1}y_1^{b_1}\cdots x_d^{a_d}y_d^{b_d} > x_1^{a_1'}y_1^{b_1'}\cdots x_d^{a_d'}y_d^{b_d'},
\]
if $(a_1,b_1,\ldots,a_d,b_d)>(a_1',b_1',\ldots,a_d',b_d')$ if the usual lexicographic order.
If $j_b<k_b$, then the leading monomials of $f_{j_b}(x_{j_b})y_{k_b}$ and $f_{k_b}(x_{k_b})y_{j_b}$ are
$x_{j_b}^{m_{j_b}}y_{k_b}$ and $y_{j_b}x_{k_b}^{m_{k_b}}$. Hence $x_{j_b}^{m_{j_b}}y_{k_b}>y_{j_b}x_{k_b}^{m_{k_b}}$
and the leading monomial of (\ref{image of normal word}) is
\begin{equation}\label{leading monomial of pi v}
\text{lead}(\pi(v))=x_1^{a_1}\cdots x_d^{a_d}x_{j_1}^{m_{j_b}}y_{k_1}\cdots x_{j_1}^{m_{j_q}}y_{k_q}.
\end{equation}
We shall show that we can recover the normal word $v$ from the leading monomial $\text{lead}(\pi(v))$.
Hence the monomials $\text{lead}(\pi(v))$ are pairwise different and
the polynomials $\pi(v)$ are linearly independent.
By Lemma \ref{spanning set of algebra of constants} this implies that
$\pi(v)$ form a basis of the vector space $K[X_d,Y_d]^{\Delta}$.

It is sufficient to compare those $\pi(v)$ which are of the same degree in $Y_d$.
As in \cite{DML} we use induction on the degree with respect to $Y_d$.
Since $\deg_{U_d}(v)=\deg_{Y_d}(\text{lead}(\pi(v)))$,
if $\deg_{Y_d}(\text{lead}(\pi(v)))=0$, then $\deg_{U_d}(v)=0$ and $v$ coincides with $\text{lead}(\pi(v))$.
Now let $\deg_{Y_d}(\text{lead}(\pi(v)))>0$.
We rewrite (\ref{leading monomial of pi v}) in the form
\begin{equation}\label{canonical leader of pi v}
\text{lead}(\pi(v))=x_1^{p_1}\cdots x_k^{p_k}y_k^{q_k}x_{k+1}^{p_{k+1}}\cdots x_d^{p_d}y_d^{q_d},\quad q_k>0.
\end{equation}
The variable $y_k$ participates in $\text{lead}(\pi(v))$ because $u_{jk}$ participates in $v$ from (\ref{normal monomials})
for some $j<k$ and $p_j\geq m_j$. We choose the maximal $i$ such that $i<k$ and $p_i\geq m_i$.
Lemma \ref{spanning set of algebra of constants}
implies that $i$ does not belong to the ideal $(j,k)$. Hence $i=j$ and the factor $u_{jk}=u_{ik}$ of $v$ is uniquely determined by
the form (\ref{canonical leader of pi v}) of $\text{lead}(\pi(v))$. In this way
$v=u_{ik}v_1$ for some normal word $v_1\in K[X_d,U_d]$ and
\[
\text{lead}(\pi(v))=x_i^{m_i}y_k\text{lead}(\pi(v_1)).
\]
By inductive arguments $\text{lead}(\pi(v_1))$ determines $v_1$ and this completes the proof.
\end{proof}

The following theorem is the main result of the paper.

\begin{theorem}
Let $\Delta$ be the derivation of $K[X_d,Y_d]$ defined by
\[
\Delta=\sum_{i=1}^df_i(x_i)\frac{\partial}{\partial y_i},
\]
where $f_i(x_i)$ are polynomials of positive degree. Then

{\rm (i)} The algebra of constants $K[X_d,Y_d]^{\Delta}$ has the presentation
\[
K[X_d,Y_d]^{\Delta}\cong K[X_d,U_d]/(R,S),
\]
where
\[
U_d=\{u_{ij}\mid 1\leq i<j\leq d\}
\]
and the ideal $(R,S)$ is generated by $R$ and $S$ from {\rm (\ref{relations between U})}
and {\rm (\ref{relations between X and U})}.

{\rm (ii)} The set $R\cup S$ is a reduced Gr\"obner basis of the ideal $(R,S)$
with respect to the DILL order of $K[X_d,U_d]$.

{\rm (iii)} The basis of the vector space $K[X_d,Y_d]^{\Delta}$
consists of the images in $K[X_d,Y_d]^{\Delta}$
of the normal words $v$ from Lemma {\rm \ref{spanning set of algebra of constants}}.
\end{theorem}

\begin{proof}
By Lemma \ref{relations between generators} $R\cup S$ belongs to the kernel of the epimorphism $\pi$
from (\ref{epimorphism pi}). A subset of the ideal $\ker(\pi)$ is its Gr\"obner basis if and only if
the set of normal words forms a basis of the factor algebra $K[X_d,U_d]/\ker(\pi)\cong K[X_d,Y_d]^{\Delta}$.
Hence the statements (i), (ii), and (iii) follow immediately from Lemma \ref{basis of normal elements}.
Since the leading terms of the relations $r(i,j,k,l)\in R$ and $s(i,j,k)\in S$ with respect to the DILL order do not divide
the monomials participating in the other relations in $R$ and $S$, we conclude that the Gr\"obner basis $R\cup S$ is reduced.
\end{proof}

\end{document}